%% file: Complete_reducibility,_Kulshammer_s_question,_conjugacy_classees.tex
\documentclass{article}

\usepackage{amssymb, amsmath, amsthm, tikz, verbatim, graphicx,lmodern,indentfirst,enumerate,color}

\usepackage[titletoc,page]{appendix}
\usepackage[margin=1.4in]{geometry}

\newtheorem{thm}{Theorem}[section]
\newtheorem{lem}[thm]{Lemma}
\newtheorem{prop}[thm]{Proposition}

\newtheorem{con}[thm]{Conjecture}

\theoremstyle{definition}

\theoremstyle{definition}
\newtheorem{question}[thm]{Question}

\theoremstyle{definition}
\newtheorem{defn}[thm]{Definition}

\theoremstyle{remark}
\newtheorem{rem}[thm]{Remark}

\numberwithin{equation}{section}

\newcommand*\circled[1]{\tikz[baseline=(char.base)]{
            \node[shape=circle,draw,inner sep=0pt,minimum size=5mm] (char) {#1};}}

\newcommand{\rootsD}[5]{\circled{#1}
            \begin{tabular}{ccc}
             &#2& \\
             #3&#4&#5
             \end{tabular}}
\makeatletter
\newcommand{\rmnum}[1]{\romannumeral #1}
\newcommand{\Rmnum}[1]{\expandafter\@slowromancap\romannumeral #1@}
\makeatother
\begin{document}

\input{title}

\input{Introduction}

\input{Preliminaries}
\input{G-cr}

\input{G-crM-cr}

\input{Kulshammer}
\input{conjugacy}

\input{acknowledgements}
\bibliography{mybib}

\end{document}

%% file: title.tex
\title{Complete reducibility, K\"ulshammer's question, conjugacy classes: a $D_4$ example}
\author{Tomohiro Uchiyama\\
National Center for Theoretical Sciences, Mathematics Division\\
No.~1, Sec.~4, Roosevelt Rd., National Taiwan University, Taipei, Taiwan\\
\texttt{email:t.uchiyama2170@gmail.com}}
\date{\today}
\maketitle 

\begin{abstract}
Let $k$ be a nonperfect separably closed field. Let $G$ be a connected reductive algebraic group defined over $k$. We study rationality problems for Serre's notion of complete reducibility of subgroups of $G$. In particular, we present a new example of subgroup $H$ of $G$ of type $D_4$ in characteristic $2$ such that $H$ is $G$-completely reducible but not $G$-completely reducible over $k$ (or vice versa). This is new: all known such examples are for $G$ of exceptional type. We also find a new counterexample for K\"ulshammer's question on representations of finite groups for $G$ of type $D_4$. A problem concerning the number of conjugacy classes is also considered. The notion of nonseparable subgroups plays a crucial role in all our constructions.  
\end{abstract}

\noindent \textbf{Keywords:} algebraic groups, complete reducibility, rationality, geometric invariant theory, representations of finite groups, conjugacy classes 

%% file: Introduction.tex
\section{Introduction}
Let $k$ be a field. Let $\overline k$ be an algebraic closure of $k$. Let $G$ be a connected affine algebraic $k$-group: we regard $G$ as a $\overline k$-defined algebraic group together with a choice of $k$-structure in the sense of Borel~\cite[AG.~11]{Borel-AG-book}. We say that $G$ is \emph{reductive} if the unipotent radical $R_u(G)$ of $G$ is trivial. Throughout, $G$ is always a connected reductive $k$-group. In this paper, we continue our study of rationality problems for complete reducibility of subgroups of $G$~\cite{Uchiyama-Nonperfect-pre},~\cite{Uchiyama-Nonperfectopenproblem-pre}. By a subgroup of $G$ we mean a (possibly non-$k$-defined) closed subgroup of $G$. Following Serre~\cite[Sec.~3]{Serre-building}:
\begin{defn}
A subgroup $H$ of $G$ is called \emph{$G$-completely reducible over $k$} (\emph{$G$-cr over $k$} for short) if whenever $H$ is contained in a $k$-defined parabolic subgroup $P$ of $G$, then $H$ is contained in a $k$-defined Levi subgroup of $P$. In particular if $H$ is not contained in any proper $k$-defined parabolic subgroup of $G$, $H$ is called \emph{$G$-irreducible over $k$} (\emph{$G$-ir over $k$} for short). 
\end{defn}

So far, most studies on complete reducibility is for complete reducibility over $\overline k$ only; see~\cite{Liebeck-Seitz-memoir},~\cite{Stewart-nonGcr},~\cite{Thomas-irreducible-JOA} for example. We say that a subgroup $H$ of $G$ is $G$-cr if it is $G$-cr over $\overline k$. Not much is known on complete reducibility over $k$ (especially for nonperfect $k$) except a few theoretical results and important examples; see~\cite[Sec.~5]{Bate-geometric-Inventione},~\cite{Bate-cocharacter-Arx},~\cite{Uchiyama-Nonperfect-pre},~\cite{Uchiyama-Nonperfectopenproblem-pre}. In~\cite[Thm.~1.10]{Uchiyama-Separability-JAlgebra},~\cite[Thm.~1.8]{Uchiyama-Classification-pre},~\cite[Thm.~1.2]{Uchiyama-Nonperfect-pre},~\cite[Sec.~6]{Bate-separability-TransAMS}, Bate et al.~and we found several examples of $k$-subgroups of $G$ that are $G$-cr over $k$ but not $G$-cr (or vice versa). All these examples are for $G$ of exceptional type ($E_6$, $E_7$, $E_8$, $G_2$) in $p=2$ and constructions are very intricate. The first main result in this paper is the following:  

\begin{thm}\label{D4example}
Let $k$ be a nonperfect separably closed field of characteristic $2$. Let $G$ be a simple $k$-group of type $D_4$. Then there exists a $k$-subgroup $H$ of $G$ that is $G$-cr over $k$ but not $G$-cr (or vice versa).
\end{thm}

A few comments are in order. First, one can embed $D_4$ inside $E_6$, $E_7$ or $E_8$ as a Levi subgroup. Since a subgroup contained in a $k$-Levi subgroup $L$ of $G$ is $G$-cr over $k$ if and only if it is $L$-cr over $k$ (Proposition~\ref{G-cr-L-cr}), one might argue that our ``new example'' is not really new. However we have checked that our example is different from any example in~\cite[Thm.~1.10]{Uchiyama-Separability-JAlgebra},~\cite[Thm.~1.8]{Uchiyama-Classification-pre},~\cite[Thm.~1.2]{Uchiyama-Nonperfect-pre}. So this is the first such example for classical $G$. Second, the non-perfectness of $k$ is essential in Theorem~\ref{D4example} in view of the following~\cite[Thm.~1.1]{Bate-separable-Paris}:
\begin{prop}\label{paris}
Let $H$ be a subgroup of $G$. Then $H$ is $G$-cr over $k$ if and only if $H$ is $G$-cr over $k_s$ (where $k_s$ is a separable closure of $k$). 
\end{prop}
So in particular if $k$ is perfect, a subgroup of $G$ is $G$-cr over $k$ if and only if it is $G$-cr. Proposition~\ref{paris} is deep: it depends on the recently proved $50$-years-old \emph{center conjecture of Tits} (see Conjecture~\ref{centerconjecture}) in spherical buildings~\cite{Serre-building},~\cite{Tits-colloq},~\cite{Weiss-center-Fourier}. Third, the $k$-definedness of $H$ in Theorem~\ref{D4example} is important. 
Actually it is not difficult to find a $\overline k$-subgroup with the desired property. For our construction of a $k$-defined subgroup $H$, it is essential for $H$ to be $\emph{nonseparable}$ in $G$. We write $\textup{Lie}(G)$ or $\mathfrak g$ for the Lie algebra of $G$. Recall~\cite[Def.~1.1]{Bate-separability-TransAMS}:
\begin{defn}
A subgroup $H$ of $G$ is \emph{nonseparable} if the dimension of $\textup{Lie}(C_G(H))$ is strictly smaller than the dimension of $\mathfrak{c}_{\mathfrak{g}}(H)$ (where $H$ acts on $\mathfrak g$ via the adjoint action). In other words, the scheme-theoretic centralizer of $H$ in $G$ (in the sense of~\cite[Def.~A.1.9]{Conrad-pred-book}) is not smooth. 
\end{defn} 
We exhibit the importance of nonseparability of $H$ in the proof of Theorem~\ref{D4example}. Proper nonseparable $k$-subgroups of $G$ are hard to find, and only handful examples are known~\cite[Sec.~7]{Bate-separability-TransAMS},~\cite[Thm.~1.10]{Uchiyama-Separability-JAlgebra}~\cite[Thm.~1.8]{Uchiyama-Classification-pre},~\cite[Thm.~1.2]{Uchiyama-Nonperfect-pre}. It is known that if $p$ is very good for $G$, every subgroup of $G$ is separable~\cite[Thm.~1.2]{Bate-separability-TransAMS}. Thus, to find a nonseparable subgroup we are forced to work in small $p$. See~\cite{Bate-separability-TransAMS},~\cite{Herpel-smoothcentralizerl-trans} for more on separability. 

In the rest of this section we assume $k$ is algebraically closed. In~\cite{Uchiyama-Separability-JAlgebra}, we asked:
\begin{question}\label{G-cr-M-cr-Q}
Let $H\leq M\leq G$ be a triple of reductive groups with $G$ and $M$ connected. If $H$ is $G$-cr then it is $M$-cr (and vice versa)?
\end{question} 
In general, the answer is no in either direction. It is easy to find a counterexample for the reverse direction: take $H=M=PGL_2$ and $G=SL_3$ in $p=2$ and $H$ sits inside $G$ via the adjoint representation. For more counterexamples, see~\cite{Liebeck-Seitz-memoir},~\cite{Stewart-nonGcr}. A counterexample for the forward direction is hard to find and only a handful such examples are known~\cite[Thm.~1.1]{Uchiyama-Separability-JAlgebra},~\cite[Thm.~1.2]{Uchiyama-Classification-pre},~\cite[Sec.~6]{Bate-separability-TransAMS}. All these examples are for $G$ of exceptional type ($E_6$, $E_7$, $E_8$, $G_2$) in $p=2$. Here is our second main result:
\begin{thm}\label{G-cr-M-cr}
Let $k$ be of characteristic $2$. Let $G$ be simple and of type $D_4$. Then there exists a pair of reductive subgroups $H<M$ of $G$ such that $(G,M)$ is a reductive pair and $H$ is $G$-cr but not $M$-cr.
\end{thm}
Recall that a pair of reductive groups $G$ and $M$ is called a \emph{reductive pair} if $\textup{Lie} M$ is an $M$-module direct summand of $\mathfrak{g}$. See~\cite{Goodbourn-reductivepairs} for more on reductive pairs. For our construction, nonseparablity of $H$ is essential~\cite[Thm.~1.4]{Bate-separability-TransAMS}: 
\begin{prop}
Suppose that $(G,M)$ is a reductive pair. Let $H$ be a subgroup of $M$ such that $H$ is separable in $G$. If $H$ is $G$-cr then $H$ is $M$-cr. 
\end{prop}

Now we move on to a problem with a slightly different flavor. Let $\Gamma$ be a finite group. By a representation of $\Gamma$ in a reductive group $G$, we mean a homomorphism from $\Gamma$ to $G$. We write $\textup{Hom}(\Gamma, G)$ for the set of representations $\rho$ of $\Gamma$ in $G$. The group $G$ acts on $\textup{Hom}(\Gamma, G)$ by conjugation. Let $\Gamma_p$ be a Sylow $p$-subgroup of $G$. In \cite[Sec.~2]{Kulshammer-Donovan-Israel}, K\"ulshammer asked: 
\begin{question}\label{KulshammerQ}
Let $G$ be a reductive algebraic group defined over an algebraically closed field of characteristic $p$. Let $\rho_p \in \textup{Hom}(\Gamma_p, G)$. Then are there only finitely many representations $\rho \in \textup{Hom}(\Gamma, G)$ such that $\rho\mid_{\Gamma_p}$ is $G$-conjugate to $\rho_p$? 
\end{question}
It is known that in general the answer is no. Two counterexamples are known: one in $G$ of type $G_2$~\cite{Bate-QuestionOfKulshammer} and the other in $G$ of type $E_6$~\cite[Thm.~1.14]{Uchiyama-Classification-pre} (both in $p=2$). The third main result in this paper is 
\begin{thm}\label{thmKul}
Let $k$ be of characteristic $2$. Let $G$ be simple of type $D_4$. Then there exists a finite group $\Gamma$ with a Sylow $2$-subgroup $\Gamma_2$ and representations $\rho_a\in \textup{Hom}(\Gamma,G)$ for $a\in k$ such that $\rho_a$ is not conjugate to $\rho_b$ for $a\neq b$ but the restrictions $\rho_a\mid_{\Gamma_2}$ are not pairwise conjugate for all $a\in k$. 
\end{thm} 
We note that nonseparability plays a crucial role in the proof of Theorem~\ref{thmKul}. 
In this paper, the reader will see that seemingly unrelated Questions~\ref{G-cr-M-cr-Q} and~\ref{KulshammerQ} (and the rationality problems for $G$-complete reducibility above and the problem on conjugacy classes below) are related: all our main results concerning these problems (Theorems~\ref{D4example},~\ref{G-cr-M-cr},~\ref{thmKul},~\ref{conjugacy-counterexample}) are based on the same mechanism (nonseparability plus some modifications). However, it is not completely clear yet (at least to the author) how exactly these problems are related. The main purpose of this paper is to give a chance for the reader to look at these problems all in once with a relatively easy example in $G$ of type $D_4$ to stimulate further research on relations between these problems.  

Finally we consider a problem on the number of conjugacy classes. Given $n\in {\mathbb N}$, we let $G$ act on $G^n$ by simultaneous conjugation:
$
g\cdot(g_1, g_2, \ldots, g_n) = (g g_1 g^{-1}, g g_2 g^{-1}, \ldots, g g_n g^{-1}). 
$
In \cite{Slodowy-book}, Slodowy proved the following result, applying Richardson's beautiful tangent space argument~\cite[Sec.~3]{Richardson-Conjugacy-Ann},~\cite[Lem.~3.1]{Richardson-orbits-BullAustralian}. 
\begin{prop}\label{conjugacy}
Let $M$ be a reductive subgroup of a reductive algebraic group $G$ defined over an algebraically closed field $k$. Let $n\in {\mathbb N}$, let $(m_1, \ldots, m_n)\in M^n$ and let $H$ be the subgroup of $M$ generated by $m_1, \ldots, m_n$. Suppose that $(G, M)$ is a reductive pair and that $H$ is separable in $G$. Then the intersection $G\cdot (m_1, \ldots, m_n)\cap M^n$ is a finite union of $M$-conjugacy classes. 
\end{prop}

Proposition~\ref{conjugacy} has many consequences; see~\cite{Bate-geometric-Inventione}, \cite{Slodowy-book}, and \cite[Sec.~3]{Vinberg-invariants-JLT} for example. Here is our main result on conjugacy classes:

\begin{thm}\label{conjugacy-counterexample}
Let $k$ be of characteristic $2$. Let $G$ be simple of type $D_4$. Let $M$ be the subsystem subgroup of type $A_1 A_1 A_1 A_1$. Then there exists $N\in \mathbb{N}$ and a tuple $\mathbf{m}\in M^N$ such that $G\cdot \bold{m} \cap M^N$ is an infinite union of $M$-conjugacy classes. 
\end{thm} 

Here is the structure of the paper. In Section 2, we set out the notation and show some preliminary results. Then in Section 3, we prove our first main result (Theorem~\ref{D4example}) concerning a rationality problem for complete reducibility. In Section 4, we prove some rationality result (Theorem~\ref{centralizerA}) related to the center conjecture. In Section 5, we give a short proof for our second main result on complete reducibility (Theorem~\ref{G-cr-M-cr}) using a recent result from Geometric Invariant Theory (Proposition~\ref{unipotentconjugate}). Then in Section 6, we prove Theorem~\ref{thmKul} giving a new counterexample to the question of K\"ulshammer. Finally in Section 7 we consider a problem on conjugacy classes and prove Theorem~\ref{conjugacy-counterexample}.

%% file: Preliminaries.tex
\section{Preliminaries}
Throughout, we denote by $k$ a separably closed field. Our references for algebraic groups are~\cite{Borel-AG-book},~\cite{Borel-Tits-Groupes-reductifs},~\cite{Conrad-pred-book},~\cite{Humphreys-book1}, and~\cite{Springer-book}. 

Let $H$ be a (possibly non-connected) affine algebraic group. We write $H^{\circ}$ for the identity component of $H$. We write $[H,H]$ for the derived group of $H$. A reductive group $G$ is called \emph{simple} as an algebraic group if $G$ is connected and all proper normal subgroups of $G$ are finite. We write $X_k(G)$ and $Y_k(G)$ ($X(G)$ and $Y(G)$) for the set of $k$-characters and $k$-cocharacters ($\overline k$-characters and $\overline k$-cocharacters) of $G$ respectively. For $\overline k$-characters and $\overline k$-cocharacters $G$ we simply say characters and cocharacters of $G$. 

Fix a maximal $k$-torus $T$ of $G$ (such a $T$ exists by~\cite[Cor.~18.8]{Borel-AG-book}). Then $T$ splits over $k$ since $k$ is separably closed. Let $\Psi(G,T)$ denote the set of roots of $G$ with respect to $T$. We sometimes write $\Psi(G)$ for $\Psi(G,T)$. Let $\zeta\in\Psi(G)$. We write $U_\zeta$ for the corresponding root subgroup of $G$. We define $G_\zeta := \langle U_\zeta, U_{-\zeta} \rangle$. Let $\zeta, \xi \in \Psi(G)$. Let $\xi^{\vee}$ be the coroot corresponding to $\xi$. Then $\zeta\circ\xi^{\vee}:\overline k^{*}\rightarrow \overline k^{*}$ is a $k$-homomorphism such that $(\zeta\circ\xi^{\vee})(a) = a^n$ for some $n\in\mathbb{Z}$.
Let $s_\xi$ denote the reflection corresponding to $\xi$ in the Weyl group of $G$. Each $s_\xi$ acts on the set of roots $\Psi(G)$ by the following formula~\cite[Lem.~7.1.8]{Springer-book}:
$
s_\xi\cdot\zeta = \zeta - \langle \zeta, \xi^{\vee} \rangle \xi. 
$
\noindent By \cite[Prop.~6.4.2, Lem.~7.2.1]{Carter-simple-book} we can choose $k$-homomorphisms $\epsilon_\zeta : \overline k \rightarrow U_\zeta$  so that 
$
n_\xi \epsilon_\zeta(a) n_\xi^{-1}= \epsilon_{s_\xi\cdot\zeta}(\pm a)
            \text{ where } n_\xi = \epsilon_\xi(1)\epsilon_{-\xi}(-1)\epsilon_{\xi}(1).  \label{n-action on group}
$

The next result~\cite[Prop.~1.12]{Uchiyama-Nonperfect-pre} shows complete reducibility behaves nicely under central isogenies. In this paper we do not specify the isogeny type of $G$. (Our argument works for $G$ of any isogeny type anyway.) Note that if $k$ is algebraically closed, the centrality assumption for $f$ is not necessary in Proposition~\ref{isogeny}. 
\begin{defn}
Let $G_1$ and $G_2$ be reductive $k$-groups. A $k$-isogeny $f:G_1\rightarrow G_2$ is \emph{central} if $\textup{ker}\,df_1$ is central in $\mathfrak{g_1}$ where $\textup{ker}\,df_1$ is the differential of $f$ at the identity of $G_1$ and $\mathfrak{g_1}$ is the Lie algebra of $G_1$. 
\end{defn}
\begin{prop}\label{isogeny}
Let $G_1$ and $G_2$ be reductive $k$-groups. Let $H_1$ and $H_2$ be subgroups of $G_1$ and $G_2$ be subgroups of $G_1$ and $G_2$ respectively. Let $f:G_1 \rightarrow G_2$ be a central $k$-isogeny. 
\begin{enumerate}
\item{If $H_1$ is $G_1$-cr over $k$, then $f(H_1)$ is $G_2$-cr over $k$.}
\item{If $H_2$ is $G_2$-cr over $k$, then $f^{-1}(H_2)$ is $G_1$-cr over $k$.} 
\end{enumerate}
\end{prop}

The next result~\cite[Thm.~1.4]{Bate-cocharacterbuildings-Arx} is used repeatedly to reduce problems on $G$-complete reducibility to those on $L$-complete reducibility where $L$ is a Levi subgroup of $G$. 

\begin{prop}\label{G-cr-L-cr}
Suppose that a subgroup $H$ of $G$ is contained in a $k$-defined Levi subgroup of $G$. Then $H$ is $G$-cr over $k$ if and only if it is $L$-cr over $k$. 
\end{prop}

We recall characterizations of parabolic subgroups, Levi subgroups, and unipotent radicals in terms of cocharacters of $G$~\cite[Prop.~8.4.5]{Springer-book}. These characterizations are essential to translate results on complete reducibility into the language of GIT; see~\cite{Bate-geometric-Inventione},~\cite{Bate-uniform-TransAMS} for example. 

\begin{defn}
Let $X$ be a affine $k$-variety. Let $\phi : \overline k^*\rightarrow X$ be a $k$-morphism of affine $k$-varieties. We say that $\displaystyle\lim_{a\rightarrow 0}\phi(a)$ exists if there exists a $k$-morphism $\hat\phi:\overline k\rightarrow X$ (necessarily unique) whose restriction to $\overline k^{*}$ is $\phi$. If this limit exists, we set $\displaystyle\lim_{a\rightarrow 0}\phi(a) = \hat\phi(0)$.
\end{defn}

\begin{defn}\label{R-parabolic}
Let $\lambda\in Y(G)$. Define
$
P_\lambda := \{ g\in G \mid \displaystyle\lim_{a\rightarrow 0}\lambda(a)g\lambda(a)^{-1} \text{ exists}\}, $\\
$L_\lambda := \{ g\in G \mid \displaystyle\lim_{a\rightarrow 0}\lambda(a)g\lambda(a)^{-1} = g\}, \,
R_u(P_\lambda) := \{ g\in G \mid  \displaystyle\lim_{a\rightarrow0}\lambda(a)g\lambda(a)^{-1} = 1\}. 
$
\end{defn}
Then $P_\lambda$ is a parabolic subgroup of $G$, $L_\lambda$ is a Levi subgroup of $P_\lambda$, and $R_u(P_\lambda)$ is the unipotent radical of $P_\lambda$. If $\lambda$ is $k$-defined, $P_\lambda$, $L_\lambda$, and $R_u(P_\lambda)$ are $k$-defined~\cite[Sec.~2.1-2.3]{Richardson-conjugacy-Duke}. Any $k$-defined parabolic subgroups and $k$-defined Levi subgroups of $G$ arise in this way since $k$ is separably closed. It is well known that $L_\lambda = C_G(\lambda(\overline k^*))$. Note that $k$-defined Levi subgroups of a $k$-defined parabolic subgroup $P$ of $G$ are $R_u(P)(k)$-conjugate~\cite[Lem.~2.5(\rmnum{3})]{Bate-uniform-TransAMS}. Let $M$ be a reductive $k$-subgroup of $G$. Then, there is a natural inclusion $Y_k(M)\subseteq Y_k(G)$ of $k$-cocharacter groups. Let $\lambda\in Y_k(M)$. We write $P_\lambda(G)$ or just $P_\lambda$ for the parabolic subgroup of $G$ corresponding to $\lambda$, and $P_\lambda(M)$ for the parabolic subgroup of $M$ corresponding to $\lambda$. It is clear that $P_\lambda(M) = P_\lambda(G)\cap M$ and $R_u(P_\lambda(M)) = R_u(P_\lambda(G))\cap M$. 

Recall the following geometric characterization for complete reducibility via GIT~\cite{Bate-geometric-Inventione}. Suppose that a subgroup $H$ of $G$ is generated by $n$-tuple ${\bf h}=(h_1,\cdots, h_n)$ of $G$, and $G$ acts on ${\bf h}$ by simultaneous conjugation. 
\begin{prop}\label{geometric}
A subgroup $H$ of $G$ is $G$-cr if and only if the $G$-orbit $G\cdot {\bf h}$ is closed. 
\end{prop}
Combining Proposition~\ref{geometric} and a recent result from GIT~\cite[Thm.~3.3]{Bate-uniform-TransAMS} we have
\begin{prop}\label{unipotentconjugate}
Let $H$ be a subgroup of $G$. Let $\lambda\in Y(G)$. Suppose that ${\bf h'}:=\lim_{a\rightarrow 0}\lambda(a)\cdot {\bf h}$ exists. If $H$ is $G$-cr, then ${\bf h'}$ is $R_u(P_\lambda)$-conjugate to ${\bf h}$. 
\end{prop}


%% file: G-cr.tex
\section{$G$-cr vs $G$-cr over $k$ (Proof of Theorem~\ref{D4example})}
Let $G$ be a simple algebraic group of type $D_4$ defined over a nonperfect field $k$ of characteristic $2$. Fix a maximal $k$-torus of $G$ and a $k$-defined Borel subgroup of $G$. let $\Psi(G)=\Psi(G,T)$ be the set of roots corresponding to $T$, and $\Psi(G)^{+}=\Psi(G,B,T)$ be the set of positive roots of $G$ corresponding to $T$ and $B$. The following Dynkin diagram defines the set of simple roots of $G$.
\begin{figure}[h]
                \centering
                \scalebox{0.7}{
                \begin{tikzpicture}
                \draw (1,0) to (2,0);
                \draw (2,0) to (2,0.85);
                \draw (2,0) to (3,0);
                \fill (1,0) circle (1mm);
                \fill (2,0) circle (1mm);
                \fill (2,0.85) circle (1mm);
                \fill (3,0) circle (1mm);
                \draw[below] (1,-0.2) node{$\alpha$};
                \draw[below] (2,-0.2) node{$\beta$};
                \draw[below] (2.4,1) node{$\gamma$};
                \draw[below] (3,-0.2) node{$\delta$};
               \end{tikzpicture}}
\end{figure}

We label $\Psi(G)^{+}$ in the following. The corresponding negative roos are defined accordingly. Note that Roots 1, 2, 3, 4 correspond to $\alpha$, $\gamma$, $\delta$, $\beta$ respectively.
\begin{table}[!h]
\begin{center}
\scalebox{0.7}{
\begin{tabular}{cccccc}
\rootsD{1}{0}{1}{0}{0}&\rootsD{2}{1}{0}{0}{0}&\rootsD{3}{0}{0}{0}{1}&\rootsD{4}{0}{0}{1}{0}&\rootsD{5}{0}{1}{1}{0}&\rootsD{6}{1}{0}{1}{0}\\
\rootsD{7}{0}{0}{1}{1}&\rootsD{8}{1}{1}{1}{0}&\rootsD{9}{0}{1}{1}{1}&\rootsD{10}{1}{0}{1}{1}&\rootsD{11}{1}{1}{1}{1}&\rootsD{12}{1}{1}{2}{1}\\
\end{tabular}
}
\end{center}
\end{table}   
Let
$\lambda:=(\alpha+2\beta+\gamma+\delta)^{\vee}=\alpha^{\vee}+2\beta^{\vee}+\gamma^{\vee}+\delta^{\vee}$. 
Then 
$
P_\lambda=\langle T, U_{\zeta}\mid \zeta\in \Psi(G)^{+}\cup \{-1,-2,-3\} \rangle,
L_\lambda=\langle T, U_{\zeta}\mid \zeta\in \{\pm 1,\pm 2,\pm 3\} \rangle,
R_u(P_\lambda)=\langle U_{\zeta} \mid \zeta \in \Psi(G)^{+}\backslash \{1, 2, 3\} \rangle.
$
Let $a\in k\backslash k^2$. Pick $b\in k^{*}$ with $b^3=1$ and $b\neq 1$. Let $v(\sqrt a):=\epsilon_{4}(\sqrt a)\epsilon_{11}(\sqrt a)\in R_u(P_\lambda)(\overline k)$. Define
\begin{equation*}
H:=v(\sqrt a)\cdot\langle n_\alpha n_\gamma n_\delta, \; (\alpha+\gamma+\delta)^{\vee}(b) \rangle.
\end{equation*}
Here is our first main result in this section.
\begin{prop}\label{firstmain}
$H$ is $k$-defined. Moreover, $H$ is $G$-cr but not $G$-cr over $k$. 
\end{prop}
\begin{proof}
First, we have 
$
(n_\alpha n_\gamma n_\delta) \cdot (\beta) = (n_\alpha n_\gamma n_\delta) \cdot 4 = 11, \;
(n_\alpha n_\gamma n_\delta) \cdot 11 = 4.
$
Using this and the commutation relations~\cite[Lem.~32.5 and Prop.~33.3]{Humphreys-book1}, we obtain
\begin{equation*}
v(\sqrt a)\cdot (n_{\alpha} n_\gamma n_\delta)=(n_\alpha n_\gamma n_\delta) \epsilon_{12}(a).
\end{equation*}
Since $\langle 4, (\alpha+\gamma+\delta)^{\vee}\rangle=-3$, $\langle 11, (\alpha+\gamma+\delta)^{\vee}\rangle=3$, and $b^3=1$, $v(\sqrt a)$ commutes with $(\alpha+\gamma+\delta)^{\vee}(b)$. Now it is clear that $H$ is $k$-defined (since it is generated by $k$-points). 

Now we show that $H$ is $G$-cr. It is sufficient to show that $H':=v(\sqrt a)^{-1}\cdot H=\langle n_\alpha n_\gamma n_\delta,\; (\alpha+\gamma+\delta)^{\vee}(b)$ is $G$-cr since it is $G$-conjugate to $H$. Since $H'$ is contained in $L_\lambda$, by Proposition~\ref{G-cr-L-cr} it is enough to show that $H'$ is $L_\lambda$-cr. By inspection, $H'$ is $L_\lambda$-ir (this is easy since $L_\lambda=L_\alpha\times L_\gamma\times L_\delta=A_1\times A_1 \times A_1$).

Next, we show that $H$ is not $G$-cr over $k$. Suppose the contrary. Clearly $H$ is contained in $P_\lambda$ that is $k$-defined. Then there exists a $k$-defined Levi subgroup of $P_\lambda$ containing $H$. Then by~\cite[Lem.~2.5(\rmnum{3})]{Bate-uniform-TransAMS} there exists $u\in R_u(P_\lambda)(k)$ such that $H$ is contained in $u\cdot L_\lambda$. Thus $n_\alpha n_\gamma n_\delta \epsilon_{12}(a) < u\cdot L_\lambda$. So $u^{-1}\cdot (n_\alpha n_\gamma n_\delta \epsilon_{12}(a)) < L_{\lambda}$. By~\cite[Prop.~8.2.1]{Springer-book}, we set
$
u:=\prod_{\zeta\in \Psi(R_u(P_\lambda))}\epsilon_\zeta(x_\zeta).
$
Using the labelling of the positive roots above, we have $\Psi(R_u(P_\lambda))=\{4,\cdots 12\}$. We compute how $n_\alpha n_\gamma n_\delta$ acts on $\Psi(R_u(P_\lambda))$: 
\begin{equation}\label{perm}
n_\alpha n_\gamma n_\delta = (4\;11) (5\;10) (6\;9) (7\;8) (12). 
\end{equation}
Using this and the commutation relations,
\begin{alignat*}{2}
u^{-1}\cdot (n_\alpha n_\gamma n_\delta \epsilon_{12}(a))
=&n_\alpha n_\gamma n_\delta \epsilon_4(x_4+x_{11})\epsilon_{5}(x_5+x_{10})\epsilon_{6}(x_6+x_9)\epsilon_{7}(x_{7}+x_{8})\\
&\epsilon_{8}(x_7+x_8)\epsilon_{9}(x_6+x_{9})\epsilon_{10}(x_5+x_{10})\epsilon_{11}(x_4+x_{11})\\
&\epsilon_{12}(x_{4}^2+x_{5}^2+x_{6}^2+x_{7}^2 +a).
\end{alignat*}
Thus if $u^{-1}\cdot (n_\alpha\sigma \epsilon_{\alpha+2\beta+\gamma+\delta}(a)) < L_{\lambda}$ we must have
\begin{equation*}
x_4=x_{11},\; x_5=x_{10},\; x_{6}=x_{9},\; x_7=x_8,\; x_{4}^2+ x_{5}^2+x_{6}^2+x_{7}^2 +a=0.
\end{equation*}
The last equation gives $(x_4+x_5+x_6+x_7)^2=a$. This is impossible since $a\notin k^2$. We are done. 
\end{proof}

\begin{rem}\label{D4nonsep}
From the computations above we see that the curve $C(x):=\{\epsilon_{4}(x)\epsilon_{11}(x)\mid x\in \overline k\}$ is not contained in $C_G(H)$, but the corresponding element in $\textup{Lie}(G)$, that is, $e_4+e_{11}$ is contained in $\mathfrak{c}_{\mathfrak{g}}(H)$. Then the argument in the proof of~\cite[Prop.~3.3]{Uchiyama-Separability-JAlgebra} shows that $\textup{Dim}(C_G(H))$ is strictly smaller than $\textup{Dim}(\mathfrak{c}_{\mathfrak{g}}(H))$. So $H$ is non-separable in $G$. 
In fact, combining~\cite[Thm.~1.5]{Bate-cocharacter-Arx} and~\cite[Thm.~9.3]{Bate-cocharacter-Arx} we have that if a $k$-subgroup $H$ of $G$ is separable in $G$ and $H$ is $G$-cr, then it is $G$-cr over $k$. 
\end{rem}

\vspace{5mm}
Now we move on to the second main result in this section. We use the same $k$, $a$, $b$, $G$, and, $\lambda$ as above. Let $v(\sqrt a):=\epsilon_{-11}(\sqrt a)\epsilon_{-4}(\sqrt a)$. Let
\begin{equation*}
K:=v(\sqrt a)\cdot \langle n_{\alpha} n_{\gamma} n_{\delta},\; (\alpha+\gamma+\delta)^{\vee}(b)\rangle=\langle n_\alpha n_\gamma n_\delta \epsilon_{-12}(a), \;  (\alpha+\gamma+\delta)^{\vee}(b)\rangle. 
\end{equation*}
Define
\begin{equation*}
H:=\langle K, \; \epsilon_{5}(1) \rangle.
\end{equation*}

\begin{prop}\label{secondmain}
$H$ is $k$-defined. Moreover, $H$ is $G$-ir over $k$ but not $G$-cr. 
\end{prop}
\begin{proof}
$H$ is clearly $k$-defined. First, we show that $H$ is $G$-ir over $k$. Note that
\begin{equation*}
v(\sqrt a)^{-1}\cdot H = \langle n_\alpha n_\gamma n_\delta, \; (\alpha+\gamma+\delta)^{\vee}(b),\; \epsilon_{5}(1)\epsilon_{1}(\sqrt a)\rangle.
\end{equation*}
Thus we see that $v(\sqrt a)^{-1}\cdot H$ is contained in $P_\lambda$. So $H$ is contained in $v(\sqrt a)\cdot P_\lambda$. 

\begin{lem}\label{uniquepara}
$v(\sqrt a)\cdot P_\lambda$ is the unique proper parabolic subgroup of $G$ containing $H$.
\end{lem}
\begin{proof}
Suppose that $P_\mu$ is a proper parabolic subgroup of $G$ containing $v(\sqrt a)^{-1}\cdot H$. In the proof of Proposition~\ref{firstmain} we have shown that $M:=\langle n_\alpha n_\gamma n_\delta, (\alpha+\gamma+\delta)^{\vee}(b)\rangle$ is $G$-cr. Then there exists a Levi subgroup $L$ of $P_\mu$ containing $M$ since $M$ is contained in $P_\mu$. Since Levi subgroups of $P_\mu$ are $R_u(P_\mu)$-conjugate by~\cite[Lem.~2.5(\rmnum{3})]{Bate-uniform-TransAMS}, without loss, we set $L:=L_\mu$. Then $M<L_\mu=C_G(\mu(\overline k^*))$, so $\mu(\overline k^*)$ centralizes $M$. Recall that by~\cite[Thm.~13.4.2]{Springer-book}, $C_{R_u(P_\lambda)}(M)^{\circ}\times C_{L_\lambda}(M)^{\circ}\times C_{R_u(P_\lambda^{-})}(M)^{\circ}$ is an open set of $C_G(M)^{\circ}$ where $P_\lambda^{-}$ is the opposite of $P_\lambda$ containing $L_\lambda$.  
\begin{lem}\label{centralizerofM}
$C_G(M)^{\circ}=G_{12}$.
\end{lem}
\begin{proof}
First of all, from Equation~(\ref{perm}) we see that $G_{12}$ is contained in $C_G(n_\alpha n_\gamma n_\delta)$. Since $\langle 12, (\alpha+\gamma+\delta)^{\vee} \rangle=0$, $G_{12}$ is also contained in $C_G((\alpha+\gamma+\delta)^{\vee}(\overline k^*))$. So $G_{12}$ is contained in $C_G(M)$. Set $u:=\prod_{i\in \Psi(R_u(P_\lambda))}\epsilon_{i}(x_i)$ for some $x_i \in \overline k$. Using Equation (\ref{perm}) and the commutation relations, we obtain
\begin{alignat*}{2}
(n_\alpha n_\gamma n_\delta)\cdot u =& \epsilon_{4}(x_{11})\epsilon_{5}(x_{10})\epsilon_{6}(x_9)\epsilon_7(x_{8})\epsilon_8(x_{7})\epsilon_9(x_6)\epsilon_{10}(x_{5})\epsilon_{11}(x_4)\\
&\epsilon_{12}(x_4 x_{11}+x_5 x_{10} +x_6 x_9+ x_7 x_8+ x_{12}). 
\end{alignat*}
So, if $u\in C_{R_u(P_\lambda)}(n_\alpha n_\gamma n_\delta)$ we must have
$x_4=x_{11}, \; x_5=x_{10},\; x_{6}=x_{9},\; x_7=x_8$, and $x_4 x_{11}+x_5 x_{10} +x_6 x_9+ x_7 x_8=0$. But $\langle \zeta, (\alpha+\gamma+\delta)^{\vee}\rangle=-1$ for $\zeta=\{5, 6, 7\}$, so $x_5=x_6=x_7=0$ for $u\in C_{R_u(P_\lambda)}(M)$. Then 
\begin{equation*}
(n_\alpha n_\gamma n_\delta)\cdot u = \epsilon_4(x_4)\epsilon_{11}(x_4)\epsilon_{12}(x_4^2+x_{12}).
\end{equation*}
So we must have $x_4^2=0$ if $u\in C_{R_u(P_\lambda)}(M)$. Thus we conclude that $C_{R_u(P_\lambda)}(M)=U_{12}$. Similarly, we can show that $C_{R_u(P_{\lambda}^{-})}(M)=U_{-12}$. A direct computation shows that $C_{L_\lambda}(M) < T$ and $C_T(n_\alpha n_\gamma n_\delta)=(\alpha+2\beta+\gamma+\delta)^{\vee}(\overline k^*)<G_{12}$. We are done.
\end{proof}
Since $\mu(\overline k^*)$ centralizes $M$, Lemma~\ref{centralizerofM} yields $\mu(\overline k^*)<G_{12}$. Then we can set $\mu:=g\cdot (\alpha+2\beta+\gamma+\delta)^{\vee}$ for some $g\in G_{12}$. By the Bruhat decomposition, $g$ is one of the following forms:
\begin{alignat*}{2}
&(1)\;g=(\alpha+2\beta+\gamma+\delta)^{\vee}(s)\epsilon_{12}(x_1),\\
&(2)\;g=\epsilon_{12}(x_1)n_{12}(\alpha+2\beta+\gamma+\delta)^{\vee}(s)\epsilon_{12}(x_2)\\
&\textup{for some } x_1, x_2\in \overline k, s\in \overline k^*.
\end{alignat*}
We rule out the second case. Suppose $g$ is of the second form. Note that $\epsilon_{5}(1)\epsilon_{1}(\sqrt a)\in v(\sqrt a)^{-1}\cdot H< P_\mu$. 
But $P_\mu=P_{g\cdot (\alpha+2\beta+\gamma+\delta)^{\vee}}=g\cdot P_{(\alpha+2\beta+\gamma+\delta)^{\vee}}$. So it is enough to show that $g^{-1}\cdot (\epsilon_{5}(1)\epsilon_{1}(\sqrt a))\notin P_{(\alpha+2\beta+\gamma+\delta)^{\vee}}$. Since $U_{12}$ and $(\alpha+2\beta+\gamma+\delta)(\overline k^*)$ are contained in $P_{(\alpha+2\beta+\gamma+\delta)^{\vee}}$ we can assume $g=n_{12}$. We have
\begin{equation*}
n_{12}=n_\alpha n_\beta n_\alpha n_\gamma n_\beta n_\alpha n_\delta n_\beta n_\alpha n_\gamma n_\beta n_\delta \textup{ (the longest element in the Weyl group of $D_4$)}.
\end{equation*}
Using this, we can compute how $n_{12}$ acts on each root subgroup of $G$. In particular $n_{12}^{-1}\cdot U_{5}=U_{-5}$ and $n_{12}^{-1}\cdot U_{1}= U_{-1}$. Thus
\begin{alignat*}{2}
n_{12}^{-1}\cdot (\epsilon_{5}(1)\epsilon_{1}(\sqrt a)) &= \epsilon_{-5}(1)\epsilon_{-1}(\sqrt a)\notin P_{(\alpha+2\beta+\gamma+\delta)^{\vee}}.
\end{alignat*}
So $g$ must be of the first form. Then $g\in P_\lambda$. Thus $P_\mu=P_{g\cdot \lambda}=g\cdot P_\lambda=P_\lambda$. We are done.
\end{proof}

\begin{lem}\label{nonkdefined}
$v(\sqrt a)\cdot P_\lambda$ is not $k$-defined.
\end{lem}
\begin{proof}
Suppose the contrary. Since $P_\lambda$ is $k$-defined, $v(\sqrt a)\cdot P_\lambda$ is $G(k)$-conjugate to $P_\lambda$ by~\cite[Thm.~20.9]{Borel-AG-book}. Thus we can put $P_\lambda=g v(\sqrt a)\cdot P_\lambda$ for some $g\in G(k)$. So $g v(\sqrt a)\in P_\lambda$ since parabolic subgroups are self-normalizing. Then $g=pv(\sqrt a)^{-1}$ for some $p\in P_\lambda$. Thus $g$ is a $k$-point of $P_\lambda R_u(P_\lambda^{-})$. Then by the rational version of the Bruhat decomposition~\cite[Thm.~21.15]{Borel-AG-book}, there exists a unique $p'\in P_\lambda(k)$ and a unique $u'\in R_u(P_\lambda^{-})(k)$ such that $g=p' u'$. This is a contradiction since $v(\sqrt a)\notin R_u(P_\lambda^{-})(k)$. 
\end{proof} 
Now Lemmas~\ref{uniquepara},~\ref{nonkdefined} show that $H$ is $G$-ir over $k$. 

\begin{lem}\label{nonG-cr}
$H$ is not $G$-cr. 
\end{lem}
\begin{proof}
We had $C_G(M)^{\circ}=G_{12}$. Then $C_G(v(\sqrt a)^{-1}\cdot H)^{\circ}<G_{12}$ since $M<v(\sqrt a)^{-1}\cdot H$. Using the commutation relations, we see that $U_{12}< C_G(v(\sqrt a)^{-1}\cdot H)$. Note that $v(\sqrt a)^{-1}\cdot H$ contains $h:=\epsilon_{5}(1)\epsilon_{1}(\sqrt a)$ that does not commute with any non-trivial element of $U_{-12}$. Also, since $\langle 5, \lambda\rangle = 4$,  $h$ does not commute with any non-trivial element of $(\alpha+2\beta+\gamma+\delta)^{\vee}(\overline k^*)$.
Thus we conclude that $C_G(v(\sqrt a)^{-1}\cdot H)^{\circ}=U_{12}$. So $C_G(H)^{\circ}=v(\sqrt a)\cdot U_{12}$ which is unipotent. Then by the classical result of Borel-Tits~\cite[Prop.~3.1]{Borel-Tits-unipotent-invent}, we see that $C_G(H)^{\circ}$ is not $G$-cr.
Since $C_G(H)^{\circ}$ is a normal subgroup of $C_G(H)$, by~\cite[Ex.~5.20]{Bate-uniform-TransAMS}, $C_G(H)$ is not $G$-cr. Then by~\cite[Cor.~3.17]{Bate-geometric-Inventione}, $H$ is not $G$-cr. 
\end{proof}
\end{proof}

%% file: G-crM-cr.tex
\section{Tits' center conjecture}
In~\cite{Tits-colloq}, Tits conjectured the following:
\begin{con}\label{centerconjecture}
Let $X$ be a spherical building. Let $Y$ be a convex contractible simplicial subcomplex of $X$. If $H$ is an automorphism group of $X$ stabilizing $Y$, then there exists a simplex of $Y$ fixed by $H$.   
\end{con}
This so-called center conjecture of Tits was proved by case-by-case analyses by Tits, M\"{u}hlherr, Leeb, and Ramos-Cuevas~\cite{Leeb-Ramos-TCC-GFA},~\cite{Muhlherr-Tits-TCC-JAlgebra},~\cite{Ramos-centerconj-Geo}. Recently uniform proof was given in~\cite{Weiss-center-Fourier}. In relation to the theory of complete reducibility, Serre showed~\cite{Serre-building}:
\begin{prop}\label{SerreContractible}
Let $G$ be a reductive $k$-group. Let $\Delta(G)$ be the building of $G$. If $H$ is not $G$-cr, then the fixed point subcomplex $\Delta(G)^H$ is  convex and contractible. 
\end{prop}
We identify the set of proper $k$-parabolic subgroups of $G$ with $\Delta(G)$ in the usual sense of Tits~\cite{Tits-book}. Note that for a subgroup $H$ of $G$, $N_G(H)(k)$ induces an automorphism group of $\Delta(G)$ stabilizing $\Delta(G)^H$. Thus, combining the center conjecture with Proposition~\ref{SerreContractible} we obtain
\begin{prop}\label{normalizer}
If a subgroup $H$ of $G$ is not $G$-cr over $k$, then there exists a proper $k$-parabolic subgroup of $G$ containing $H$ and $N_G(H)(k)$. 
\end{prop}
Proposition~\ref{normalizer} was an essential tool to prove various theoretical results on complete reducibility over nonperfect $k$ in~\cite{Uchiyama-Nonperfectopenproblem-pre} and~\cite{Uchiyama-Nonperfect-pre}. We have asked the following in~\cite[Rem.~6.5]{Uchiyama-Nonperfectopenproblem-pre}:
\begin{question}\label{centralizerQ}
If $H<G$ is not $G$-cr over $k$, then does there exist a proper $k$-parabolic subgroup of $G$ containing $HC_G(H)$?
\end{question}
The answer is yes if $C_G(H)$ is $k$-defined (or $k$ is perfect). Since in that case the set of $k$ points are dense in $C_G(H)$ (since we assume $k=k_s$) and the result follows from Proposition~\ref{normalizer}. The main result in this section is to present a counterexample to Question~\ref{centralizerQ} when $k$ is nonperfect. 
\begin{thm}\label{centralizerA}
Let $k$ be nonperfect of characteristic $2$. Let $G$ be simple of type $D_4$. Then there exists a non-abelian $k$-subgroup $H$ of $G$ such that $H$ is not $G$-cr over $k$ but $C_G(H)$ is not contained in any proper $k$-parabolic subgroup of $G$. 
\end{thm}
\begin{rem}
Borel-Tits~\cite[Rem.~2.8]{Borel-Tits-unipotent-invent} mentioned that if $k$ is nonperfect of characteristic $2$ and $[k:k^2]>2$, there exists a $k$-plongeable unipotent element $u$ in $G$ of type $D_4$ such that $C_G(u)$ is not contained in any proper $k$-parabolic subgroup of $G$ (with no proof). Note that such $u$ generates a (cyclic) subgroup of $G$ that is not $G$-cr over $k$. (Recall that a unipotent element is called $k$-plongeable if it can be embedded in the unipotent radical of a proper $k$-parabolic subgroup of $G$~\cite{Borel-Tits-unipotent-invent}.) Theorem~\ref{centralizerA} is a nonabelian version of Borel-Tits' result. Also the assumption $[k:k^2]>2$ is not necessary here. 
\end{rem}
\begin{proof}
We keep the same notation from the previous section. Set $n:=n_\alpha n_\gamma n_\delta$,  $t:=(\alpha+\gamma+\delta)^{\vee}(b)$, and $v(\sqrt a):=\epsilon_4(\sqrt a)\epsilon_{11}(\sqrt a)$. Let $H:=\langle n_\alpha n_\gamma n_\delta \epsilon_{12}(a), (\alpha+\gamma+\delta)^{\vee}(b) \rangle$.  Then $H$ is not $G$-cr over $k$. 
We have $H':=v(\sqrt a)^{-1}\cdot H = \langle n, t \rangle$. It is clear that $C_G(H')>G_{12}$. Thus $\langle n, t, G_{12} \rangle < H'C_G(H')$. By running a similar argument as in the proof of Lemma~\ref{uniquepara} in the previous section, we find that the only proper parabolic subgroup of $G$ containing $\langle n, t, U_{12} \rangle$ is $P_{(\alpha+2\beta+\gamma+\delta)^{\vee}}$ (since $n_{12}\cdot 12 = -12$). Clearly $P_{(\alpha+2\beta+\gamma+\delta)^{\vee}}$ does not contain $G_{12}$. Therefore there is no proper parabolic subgroup of $G$ containing $H'C_G(H')$. Thus there is no proper parabolic subgroup of $G$ containing $HC_G(H)$. 
\end{proof}

\section{G-cr vs $M$-cr (Proof of Theorem~\ref{G-cr-M-cr})}
From this section we assume $k$ is algebraically closed. Let $G$ be as in the hypothesis. Let $a, b\in k^{*}$ with $b^3=1$ and $b\neq 1$. Let $H':=\langle n_\alpha n_\gamma n_\delta, (\alpha+\gamma+\delta)^{\vee}(b) \rangle$. Let $v(a):=\epsilon_4(a)\epsilon_{11}(a)$. Define
\begin{equation*}
H:=v(a)\cdot H' = \langle n_\alpha n_\gamma n_\delta \epsilon_{12}(a^2), (\alpha+\gamma+\delta)^{\vee}(b)\rangle.  
\end{equation*}
Then $H$ is $G$-cr (by the same argument as in the previous section). Now let $M:=\langle G_{\alpha}, G_{\gamma}, G_{\delta}, G_{12}\rangle\cong A_1 A_1 A_1 A_1$. 
\begin{prop}
$H$ is not $M$-cr. 
\end{prop}
\begin{proof}
Let $\lambda:=(\alpha+2\beta+\gamma+\delta)^{\vee}$. Then $H<P_\lambda(M)=\langle T, G_{\alpha}, G_{\gamma}, G_{\delta}, U_{12}\rangle$. Let $c_\lambda: P_\lambda\rightarrow L_\lambda$ be the natural projection. Let $v:=(n_\alpha n_\gamma n_\delta \epsilon_{12}(a^2), (\alpha+\gamma+\delta)^{\vee}(b))$. We have
\begin{equation*}
c_\lambda(v)=\lim_{a\rightarrow 0}\lambda(a)\cdot (n_\alpha n_\gamma n_\delta \epsilon_{12}(a^2), (\alpha+\gamma+\delta)^{\vee}(b))= (n_\alpha n_\gamma n_\delta, (\alpha+\gamma+\delta)^{\vee}(b)).
\end{equation*}
We see that $v$ is not $R_u(P_\lambda(M))$-conjugate to $c_\lambda(v)$ since $R_u(P_\lambda (M))=U_{12}$ centralizes $n_\alpha n_\gamma n_\delta$. By Proposition~\ref{unipotentconjugate}, this shows that $H$ is not $M$-cr. 
\end{proof}

%% file: Kulshammer.tex
\section{K\"ulshammer's question (Proof of Theorem~\ref{thmKul})}
Let $d\geq 5$ be odd. Let $D_{2d}$ be the dihedral group of order $2d$. Let 
\begin{equation*}
\Gamma:=D_{2d}\times C_2 =\langle r, s, z\mid r^d=s^2=z^2=1, srs^{-1}=r^{-1}, [r,z]=[s,z]=1\rangle.
\end{equation*}
Let $\Gamma_2:=\langle s, z\rangle$ (a Sylow $2$-subgroup of $\Gamma$). Let $G$ be as in the hypothesis. Choose $a, b\in k^{*}$ with $b^d=1$ and $b\neq 1$. Let $n:=n_\alpha n_\gamma n_\delta$ and $t:=(\alpha+\gamma+\delta)^{\vee}(b)$. For each $a\in k$ define $\rho_a\in \textup{Hom}(\gamma, G)$ by
\begin{equation*}
\rho_a(r)=t, \; \rho_a(s)=n \epsilon_{12}(a), \; \rho_a(z)=\epsilon_{12}(1). 
\end{equation*}
An easy computation shows that this is well-defined. 
Let $u(\sqrt a)=\epsilon_{4}(\sqrt a)\epsilon_{11}(\sqrt a)$. Then $u(\sqrt a)\cdot n = n\epsilon_{12}(a)$ and $u(\sqrt a)\cdot \epsilon_{12}(1) = \epsilon_{12}(1)$. Thus $u(\sqrt a)\cdot (\rho_0\mid_{\Gamma_2})=\rho_a\mid_{\Gamma_2}$. So $\rho_a\mid_{\Gamma_2}$ are pairwise conjugate. 

Now suppose that $\rho_a$ is conjugate to $\rho_b$. Then there exists $g\in G$ such that $g\cdot \rho_a=\rho_b$. Since $\rho_a(r)=t$, we must have $g\in C_G(t)=TG_{12}$. So let $g=hm$ for some $h\in T$ and $m\in G_{12}$. Then we have
\begin{alignat*}{2}
hnh^{-1}(hm\epsilon_{12}(a)m^{-1}h^{-1})&= hnm\epsilon_{12}(a)m^{-1}h^{-1}\\
                                                                 &= hmn\epsilon_{12}(a)m^{-1}h^{-1}\\
                                                                 &= g\cdot \rho_{a}(s)\\
                                                                 &= \rho_b(s)\\
                                                                 &= n\epsilon_{12}(b).
\end{alignat*}
Note that $hnh^{-1}\in G_\alpha G_\gamma G_\delta$ and $hm\epsilon_{12}(a)m^{-1}h^{-1}\in G_{12}$. Since $[G_\alpha G_\gamma G_\delta, G_{12}]=1$, we have $G_\alpha G_\gamma G_\delta \cap G_{12}=1$. This implies $hnh^{-1}=n$. Now an easy computation shows $h\in G_{12}$. Thus $g=hm\in G_{12}$. Since $G_{12}$ is a simple group of type $A_1$, $(n\epsilon_{12}(a), \epsilon_{12}(1))$ cannot be $G_{12}$- conjugate to $(n\epsilon_{12}(b), \epsilon_{12}(1))$ if $a\neq b$. We are done.                                                                 

%% file: conjugacy.tex
\section{Conjugacy classes (Proof of Theorem~\ref{conjugacy-counterexample})}

\begin{proof}
Let $G$ be as in the hypothesis. Let $\lambda:=(\alpha+2\beta+\gamma+\delta)^{\vee}$. Then $\Psi(R_u(P_\lambda))=\{4,\cdots, 12\}$. Using the commutation relations we have $Z(R_u(P_\lambda))=U_{12}$. Let $n:=n_\alpha n_\gamma n_\delta$. Pick $b\in k$ with $b^3=1$ and $b\neq 1$. Let $t:=(\alpha+\gamma+\delta)^{\vee}(b)$. Define $K:=\langle n, t, U_{12} \rangle$. 
By the same argument as that in the proof of~\cite[Lem.~5.1]{Uchiyama-Separability-JAlgebra} we obtain $C_{P_\lambda}(K)=C_{R_u(P_\lambda)}(K)$ (since $\langle 12, \lambda\rangle=2$). By a standard result there exists $n\in \mathbb{N}$ such that $Z=\langle z_1,\cdots, z_n \rangle$. Now let $M:=\langle L_\lambda, G_{12} \rangle$. Let ${\bf m}:=(n, t, z_1,\cdots, z_n)$ and set $N:=n+2$. Then by the similar argument to that in the proof of~\cite[Lem.~5.1]{Uchiyama-Separability-JAlgebra} yields that $G\cdot {\bf m}\cap P_\lambda(M)^N$ is an infinite union of $P_\lambda(M)$-conjugacy classes. (The crucial thing here is the existence of a curve that is tangent to $\mathfrak{c}_{\mathfrak{g}}(K)$ but not tangent to $C_G(K)$, in other words $K$ is nonseparable in $G$.) Now let $c_\lambda:P_\lambda\rightarrow L_\lambda$ be the canonical projection. Then $c_\lambda(n, t, z_1, \cdots, z_n)=(n,t)$. Since $K_0:=\langle n, t \rangle$ is $L$-ir as shown in the previous section, by~\cite[Prop.~3.5.2]{Stewart-thesis} we are done. 
\end{proof}

%% file: acknowledgements.tex
\section*{Acknowledgements}
This research was supported by a postdoctoral fellowship at the National Center for Theoretical Sciences at the National Taiwan University. 